\documentclass[11pt]{amsart}
\usepackage{amssymb,hyperref}

\newtheorem{theorem}{Theorem}[section]
\newtheorem{corollary}[theorem]{Corollary}

\newtheorem{lemma}[theorem]{Lemma}
\newtheorem{remark}[theorem]{Remark}

\theoremstyle{definition}
\newtheorem{definition}[theorem]{Definition}
\newtheorem{construction}[theorem]{Construction}

\newcommand{\Split}{\mathrm{Split}}

\newcommand{\PX}{\mathrm{PX}}
\newcommand{\cT}{\mathcal {T}}
\newcommand{\cF}{\mathcal {F}}

\def\soc{{\rm soc}}
\def\Aut{{\rm Aut}}

\def\Cay{{\rm Cay}}
\newcommand{\C}{\mathrm{C}}
\newcommand{\D}{\mathrm{D}}
\newcommand{\Q}{\mathbf{Q}}
\newcommand{\K}{\mathbf{K}}
\newcommand{\M}{\mathrm{M}}
\renewcommand{\S}{\mathrm{S}}
\newcommand{\V}{\mathrm{V}}
\newcommand{\ZZ}{\mathbb Z}

\newcommand{\cC}{\mathcal {C}}
\renewcommand{\wr}{\mathop{\mathrm{wr}}}

\def\cent#1#2{{\bf C}_{{#1}}({{#2}})}

\begin{document}

\title{Semiregular automorphisms of cubic vertex-transitive graphs} 

\author[J. Morris]{Joy Morris}
\address{Joy Morris, Department of Mathematics and Computer Science,
University of Lethbridge, Lethbridge, AB. T1K 3M4. Canada}
\email{joy@cs.uleth.ca}

\author[P. Spiga]{Pablo Spiga}
\address{Pablo Spiga,
Dipartimento di Matematica e Applicazioni, University of Milano-Bicocca, Via Cozzi 53, 20125 Milano, Italy}\email{pablo.spiga@unimib.it}

\author[G. Verret]{Gabriel Verret}
\address{Gabriel Verret, Centre for Mathematics of Symmetry and Computation, School of Mathematics and Statistics, The University of Western Australia, 35 Stirling Highway, Crawley, WA 6009, Australia. \newline
\indent Also affiliated with : UP FAMNIT, University of Primorska, Glagolja\v{s}ka 8, 6000 Koper, Slovenia.}\email{gabriel.verret@uwa.edu.au}
 
 \thanks{This work was supported in part by the first author's grant from the Natural Science and Engineering Research Council of Canada. The third author is supported by UWA as part of the Australian Research Council grant DE130101001. }
 
\subjclass[2010]{Primary 20B25, Secondary 05E18}
 \keywords{semiregular automorphism, normal quotient, cubic graph}
 
\begin{abstract}
We characterise connected cubic graphs admitting a vertex-transitive group of automorphisms with an abelian normal subgroup that is not semiregular. We illustrate the utility of this result by using it to prove that the order of a semiregular subgroup of maximum order in a vertex-transitive group of automorphisms of a connected cubic graph grows with the order of the graph, settling~\cite[Problem~6.3]{CGJKKMN}.
\end{abstract}

\maketitle

\section{Introduction}
All the graphs and groups considered in this paper are finite. A useful tool in the theory of group actions on graphs is the \emph{abelian normal quotient method}. This is used to study (and possibly classify) a family of pairs $(\Gamma,G)$  having certain additional properties, where $\Gamma$ is a finite graph and $G$ is a subgroup of the automorphism group $\Aut(\Gamma)$ of $\Gamma$. (For example, the family consisting of the pairs $(\Gamma,G)$ where $\Gamma$ is a finite $(G,s)$-arc-transitive graph, see~\cite{GLP}.)  To use this method, one  generally splits the analysis into three cases, as follows:

\begin{enumerate}
\item $G$ has no nontrivial abelian normal subgroups; \label{caseSS}
\item $G$ has an abelian normal subgroup that is not semiregular; \label{caseAbNSemi}
\item $G$ has an abelian normal subgroup that is  semiregular. \label{caseAbSemi}
\end{enumerate}

This method is inductive: cases~(\ref{caseSS}) and~(\ref{caseAbNSemi}) serve as the basis for the induction, while case~(\ref{caseAbSemi}) can be treated as a reduction. The abelian normal quotient method is a variant of the usual \emph{normal quotient method}, which already has an impressive pedigree (for example, see~\cite{GLP,GXu,PNQ,Praeger}). 

In the usual normal quotient method, one considers arbitrary normal subgroups rather than only abelian ones. Compared to this, the abelian variant trades a potentially more difficult basis of induction to obtain an easier reduction step. It seems that, in practice, this is often an advantageous trade-off and many recent papers have used this approach (see for example \cite{Dobson,Li,PSV4valent,genlost,8valent}).

We now give a few more details concerning this method. In case~(\ref{caseSS}), $G$ has trivial soluble radical. Such a group has some well-known properties: its socle is a direct product of nonabelian simple groups and the group acts faithfully on its socle by conjugation. In particular, in case~(\ref{caseSS}), the Classification of Finite Simple Groups can be brought to bear on the problem to obtain very detailed information.
 
Similarly, the situation in case~(\ref{caseAbNSemi}) is surprisingly restrictive and very strong results can often be proved under this  hypothesis. Consider, for example, the following theorem due to Praeger and Xu (the graphs which appear in the statement will be defined in Section~\ref{sec:PX}):

\begin{theorem}[{{\cite[Theorem~$1$]{PraegerXu}}}]\label{PXu}
Let $\Lambda$ be a connected $4$-valent $G$-arc-transitive graph. If $G$ has an abelian normal subgroup that is not semiregular then $\Lambda\cong \PX(2,r,s)$ for some $r\geq 3$ and $1\leq s\leq r-1$. 
\end{theorem}

Clearly, Theorem~\ref{PXu} is very useful when applying the abelian normal quotient method to $4$-valent arc-transitive graphs, as it deals with case~(\ref{caseAbNSemi}) as satisfactorily as one could hope for, that is, giving a complete classification of the possible graphs. (For examples of applications, see~\cite{PSV4valent,genlost, Verret2}.)

 One of our goals is to prove the following analogue of Theorem~\ref{PXu} for cubic vertex-transitive graphs (the graphs which appear in Theorem~\ref{CubicPXu} will be defined in Section~\ref{sec:PX}):

\begin{theorem}\label{CubicPXu}
Let $\Gamma$ be a connected cubic $G$-vertex-transitive graph.  If $G$ has an abelian normal subgroup that is not semiregular then $\Gamma$ is isomorphic to one of $\K_4$, $\K_{3,3}$, $\Q_3$ or $\S(\PX(2,r,s))$ for some $r\geq 3$ and $1\leq s\leq r-1$.
\end{theorem}

Much like Theorem~\ref{PXu} with respect to $4$-valent arc-transitive graphs, Theorem~\ref{CubicPXu} will be very useful when applying the abelian normal quotient method to cubic vertex-transitive graphs. To illustrate this usefulness, we prove the following:

\begin{theorem}\label{theorem:main1}
There exists a function $f:\mathbb{N}\to \mathbb{N}$ satisfying $f(n)\rightarrow \infty$ as $n\rightarrow\infty$ such that, if $\Gamma$ is a connected $G$-vertex-transitive cubic graph of order $n$ then $G$ contains a semiregular subgroup of order at least $f(n)$.
\end{theorem}

Theorem~\ref{theorem:main1} settles positively  the conjecture posed in~\cite[Problem~6.3]{CGJKKMN}. Note that, contrary to what is claimed in the statement of~\cite[Problem~6.3]{CGJKKMN}, the conjecture in~\cite[Problem BCC~17.12]{Cameron} (which also appeared in~\cite{CSS} as Conjecture~$2$) is actually stronger. Namely,~\cite[Problem BCC~17.12]{Cameron} strengthens~\cite[Problem~6.3]{CGJKKMN} by considering  only cyclic semiregular subgroups. 

Despite the fact that~\cite[Problem~6.3]{CGJKKMN} has a positive solution,~\cite[Problem BCC~17.12]{Cameron} was recently shown to be false  by the second author~\cite{SpigaCubic}. Note also that Theorem~\ref{theorem:main1} has appeared previously in~\cite{Li}, however the proof in that paper contains a critical mistake (in the proof of Claim~2, on the last page).

\begin{remark}{\rm In our proof
of Theorem~\ref{theorem:main1} we do not make any effort to optimise or even keep track of the most rapidly growing function $f$ satisfying the hypothesis. Our current proof shows that $f(n)$ can be taken to be $\log(\log(n))$. However, we believe this is far from best possible. In fact, we conjecture that there exists a constant $c>0$ such that $f(n)$ can be taken to be $n^{c}$. In some sense, this is best possible as it was shown in~\cite{CSS} that $f(n)\leq n^{1/3}$, for infinitely many values of $n$.}
\end{remark}

The notation used throughout this paper is standard. If $\Gamma$ is a graph and $G$ $\leq\Aut(\Gamma)$, we say that $\Gamma$ is \textit{$G$-vertex-transitive} (respectively, \textit{$G$-arc-transitive}) if $G$ acts transitively on the vertices (respectively, arcs) of $\Gamma$. If $v$ is a vertex of $\Gamma$,  the neighbourhood of $v$ is denoted by $\Gamma(v)$, the stabiliser of $v$ in $G$ is denoted by $G_v$ and $G_v^{\Gamma(v)}$ denotes the permutation group induced by $G_v$ in its action on $\Gamma(v)$.

Let $\Gamma$ be a $G$-vertex-transitive graph and let $N$ be a normal subgroup of $G$. For every vertex $v$, the $N$-orbit containing $v$ is denoted by $v^N$. The \emph{normal quotient graph} $\Gamma/N$ has the $N$-orbits on $\V(\Gamma)$ as vertices, with an edge between distinct vertices $v^N$ and $w^N$ if and only if there is an edge of $\Gamma$ between $v'$ and $w'$, for some $v' \in v^N$ and some $w' \in w^N$. Note that $G$ has an induced transitive action on the vertices of $\Gamma/N$. Moreover, it is easily seen that the valency of $\Gamma/N$ is less or equal to the valency of $\Gamma$.

The dihedral group of order $2r$ is denoted by $\D_r$. It is usually viewed as a permutation group of degree $r$ in the natural way.

The remainder of our paper is divided as follows:  in Section~\ref{sec:PX}, we define the graphs which appear in Theorems~\ref{PXu} and~\ref{CubicPXu}, prove some useful results about them, and prove Theorem~\ref{CubicPXu}.  Theorem~\ref{theorem:main1} is proved in Section~\ref{sec:last}.

\section{Praeger-Xu graphs and their split graphs}\label{sec:PX}
We first define the graphs $\PX(2,r,s)$ and prove some useful results about them.

\begin{definition}
Let $r$ and $s$ be positive integers with $r\geq 3$ and $1\leq s\leq r-1$. The graph $\PX(2,r,s)$ has vertex-set $\ZZ_2^s\times\ZZ_r$ and edge-set $\{\{(n_0,n_1,\ldots,n_{s-1},x),(n_1,\ldots,n_{s-1},n_{s},x+1)\}\mid n_i\in \ZZ_2,x\in\ZZ_r\}$. 
\end{definition}

Here is another description of these graphs that is more geometric and sometimes easier to work with. First, the graph $\PX(2, r, 1)$ is the lexicographic product of a cycle of length $r$ and an edgeless graph on two vertices. In other words, $\V(\PX(2, r, 1)) = \ZZ_2 \times \ZZ_r$ with $(u, x)$ being adjacent to $(v, y)$ if and only if $x - y \in \{-1, 1\}$. Next, a path in $\PX(2, r, 1)$ is called \emph{traversing} if it contains at most one vertex from $\ZZ_2 \times \{y\}$, for each $y \in \ZZ_r$. Finally, for $s \geq 2$, the graph $\PX(2, r, s)$ has vertex-set the set of traversing paths of $\PX(2, r, 1)$ of length $s - 1$, with two such paths being adjacent in $\PX(2, r, s)$ if and only if their union is a traversing path of length $s$ in $\PX(2, r, 1)$. 

It is not hard to see that this is equivalent to the original definition and that $\PX(2,r,s)$ is a connected $4$-valent graph with $r2^s$ vertices. Observe that there is a natural action of the wreath product $W:=\ZZ_2 \wr \D_r=\ZZ_2^r\rtimes\D_r$ as a group of automorphisms of $\PX(2, r, 1)$ with an induced faithful arc-transitive action on $\PX(2, r, s)$, for every $s$. Specifically, $W$ acts on $\V(\PX(2,r,s))=\ZZ_2^s\times\ZZ_r$ in the following way: for $g=(g_0,\ldots,g_{r-1},h)\in W$ (with $g_0,\ldots,g_{r-1}\in\ZZ_2$ and $h\in\D_r$), we have
$$(n_0,n_1,\ldots,n_{s-1},x)^g=(n_0^{g_x},n_1^{g_{x+1}},\ldots,n_{s-1}^{g_{x+s-1}},x^h), $$
where the indices are taken modulo $r$. We will also need the concept of an arc-transitive cycle decomposition, which was studied in some detail in~\cite{CycleDec}.

\begin{definition}\label{def:cycledec}
A \emph{cycle} in a graph is a connected regular subgraph of valency 2. A \emph{cycle decomposition} $\cC$ of a graph $\Lambda$  is a set of cycles in $\Lambda$ such that each edge of $\Lambda$  belongs to exactly one cycle in $\cC$. If there exists an arc-transitive group of automorphisms of $\Lambda$ that maps every cycle of $\cC$ to a cycle in $\cC$ then $\cC$ will be called \emph{arc-transitive}.
\end{definition}

\begin{construction}[{{\cite[Construction~$11$]{PSV1280}}}]\label{cons:split}
The input of this construction is a pair $(\Lambda,\cC)$, where $\Lambda$ is a $4$-valent graph and $\cC$ is an arc-transitive cycle decomposition of $\Lambda$. The output is the graph $\Split(\Lambda,\cC)$, the vertices of which are the pairs $(v,C)$ where $v\in\V(\Lambda)$, $C\in\cC$ and $v$ lies on the cycle $C$, and two vertices $(v_1,C_1)$ and $(v_2,C_2)$ are adjacent if and only if either $C_1\neq C_2$ and $v_1=v_2$, or $C_1=C_2$ and $\{v_1,v_2\}$ is an edge of $C_1$.
\end{construction}

Note that $\Split(\Lambda,\cC)$ is a cubic graph. We now consider a very important cycle decomposition of $\PX(2,r,s)$:

\begin{definition}\label{defdef}
Let $\underline{n}=(n_1,\ldots,n_{s-1})\in \ZZ_2^{s-1}$, let $x\in \ZZ_r$ and let $C_{\underline{n},x}$ be the cycle of length four of $\PX(2,r,s)$ given by 
$$((0,\underline{n},x),(\underline{n},0,x+1),(1,\underline{n},x),(\underline{n},1,x+1)).$$
Then $\cC:=\{C_{\underline{n},x}\mid \underline{n}\in \ZZ_2^{s-1},x\in\ZZ_r\}$ is a cycle decomposition of $\PX(2,r,s)$ into cycles of length four called the \emph{natural} cycle decomposition of $\PX(2,r,s)$. 
As the arc-transitive action of $\ZZ_2\wr \D_r$ on $\PX(2,r,s)$ induces a transitive action on $\cC$, we see that $\cC$ is arc-transitive.
The graph $\Split(\PX(2,r,s),\cC)$ is simply denoted by $\S(\PX(2,r,s))$.

It is not hard to see that the graph $\S(\PX(2,r,s))$ can also be described in the following way: its vertex-set is $\ZZ_2^s\times\ZZ_r\times\{+,-\}$ and its edge-set is $$\{\{(n_0,\ldots,n_{s-1},x,+),(n_1,\ldots,n_{s},x+1,-)\}\mid n_i\in \ZZ_2,x\in\ZZ_r\}~\cup\qquad$$
$$\qquad \{\{(n_0,\ldots,n_{s-1},x,+),(n_0,\ldots,n_{s-1},x,-)\}\mid n_i\in \ZZ_2,x\in\ZZ_r\}.$$
Observe that the wreath product $W:=\ZZ_2\wr\D_r=\ZZ_2^r\rtimes\D_r$ has a faithful action on $\V(\Gamma)=\ZZ_2^s\times\ZZ_r\times\{+,-\}$. Namely, for $g=(g_0,\ldots,g_{r-1},h)\in W$ (with $g_0,\ldots,g_{s-1}\in \ZZ_2$ and $h\in \D_r$), we have 
$$(n_0,n_1,\ldots,n_{s-1},x,\pm)^g=\begin{cases}(n_0^{g_x},n_1^{g_{x+1}},\ldots,n_{s-1}^{g_{x+s-1}},x^h,\pm) \textrm{ if } h\in\ZZ_r,\\
(n_0^{g_x},n_1^{g_{x+1}},\ldots,n_{s-1}^{g_{x+s-1}},x^h,\mp) \textrm{ otherwise,}\end{cases}$$
where the indices are taken modulo $r$. It is easy to check that $W$ is a vertex-transitive group of automorphisms of $\S(\PX(p,r,s))$.
\end{definition}

The graphs $\S(\PX(2,r,s))$ have appeared before in the literature, see for example~\cite[Section~3]{Dobson} and~\cite[Corollary~1.5]{PSV4valent}. In fact, most of the effort in~\cite{Dobson} is spent proving a variant of Theorem~\ref{CubicPXu}. It seems the authors were unaware of Theorem~\ref{PXu}, which might have made their work easier.

\begin{lemma}\label{boring}
Up to conjugacy in $\Aut(\PX(2,r,s))$, the natural cycle decomposition of $\PX(2,r,s)$ is the unique arc-transitive cycle decomposition of $\PX(2,r,s)$ into cycles of length four.
\end{lemma}
\begin{proof}
Let $\Lambda=\PX(2,r,s)$, let $W=\ZZ_2^r\rtimes\D_r$ and
let $\cC$ be an arbitrary arc-transitive cycle decomposition of $\Lambda$ into cycles of length four. We show that $\cC$ is conjugate to the natural cycle decomposition of $\Lambda$ under $\Aut(\Lambda)$. 

Suppose first that $r\neq 4$. In this case, we actually prove that $\cC$ is the natural cycle decomposition. By~\cite[Theorem~2.13]{PraegerXu}, we have $\Aut(\Lambda)=W$. Let $\pi$ be the canonical projection from $\V(\Lambda)=\ZZ_2^s\times\ZZ_r$ to $\ZZ_r$. 

Suppose that, for every $C\in\cC$, we have $|\pi(C)|= 2$. Let $C\in\cC$ and write $C=(v_0,v_1,v_2,v_3)$ with $v_0,v_1,v_2,v_3\in \V(\Lambda)$. Then $\pi(C)=\{x,x+1\}$ for some $x\in \ZZ_r$ and, replacing $(v_0,v_1,v_2,v_3)$ by $(v_1,v_2,v_3,v_0)$ if necessary, we may assume that $\pi(v_0)=\pi(v_2)=x$ and $\pi(v_1)=\pi(v_3)=x+1$. Thus $v_0=(n_0,n_1,\ldots,n_{s-1},x)$, $v_1=(n_1,n_2,\ldots,n_{s},x+1)$, $v_2=(1-n_0,n_1,\ldots,n_{s-1},x)$ and $v_3=(n_1,\ldots,n_{s-1},1-n_s,x+1)$, for some $n_0,\ldots,n_s\in\ZZ_2$. Replacing $(v_0,v_1,v_2,v_3)$ by $(v_2,v_3,v_0,v_1)$ and $(v_0,v_1,v_2,v_3)$ by $(v_0,v_3,v_2,v_1)$ if necessary, we may assume that $n_0=n_s=0$.  Thus $C=C_{\underline{n},x}$ where $\underline{n}=(n_1,\ldots,n_{s-1})$. Since $C$ is an arbitrary element of $\cC$ we have shown that $\cC$ is the natural cycle decomposition of $\Lambda$.

Suppose now that we have $|\pi(C)|\geq 3$ for some $C\in\cC$. In particular, $C$ contains a $2$-path $P$ such that $\pi(P)=(x,x+1,x+2)$ for some $x\in\ZZ_r$. Since $\cC$ is preserved by an arc-transitive group of automorphisms of $\Lambda$, there exists $g\in\Aut(\Lambda)$ such that $g$ acts on $C$ as a one-step rotation. As $\Aut(\Lambda)=W$, we have $g=(g_0,\ldots,g_{r-1},h)$, for some $g_0,\ldots,g_ {r-1}\in \ZZ_2$ and $h\in \D_r$. Up to replacing $g$ by its inverse, we may assume that $\pi(P^g)=(x+1,x+2,x+3)$. In particular, $h$ has order $r$. Since $C$ is a $4$-cycle and $r\neq 4$, this is a contradiction.

If $r=4$ then $1\leq s\leq 3$ and there are only three graphs to consider: $\PX(2,4,1)$, $\PX(2,4,2)$ and $\PX(2,4,3)$. The statement can then be checked case-by-case, either by hand or with the assistance of a computer.
\end{proof}

Let $\K_4$ denote the complete graph on $4$ vertices, $\K_{3,3}$ the complete bipartite graph with parts of size $3$ and $\Q_3$ the $3$-cube. We now prove Theorem~\ref{CubicPXu}, which we restate for convenience.

\smallskip

\noindent\textbf{Theorem~\ref{CubicPXu}.}\emph{
Let $\Gamma$ be a connected cubic $G$-vertex-transitive graph.  If $G$ has an abelian normal subgroup that is not semiregular then $\Gamma$ is isomorphic to one of $\K_4$, $\K_{3,3}$, $\Q_3$ or $\S(\PX(2,r,s))$ for some $r\geq 3$ and $1\leq s\leq r-1$.}
\begin{proof}
Let $v\in\V(\Gamma)$, let $N$ be an abelian normal subgroup of $G$ that is not semiregular and let $p$ be a prime dividing $|N_v|$. Note that the subgroup of $N$ generated by the elements of order $p$ is elementary abelian, is not semiregular and is characteristic in $N$, and thus normal in $G$. In particular, replacing $N$ by this subgroup, we may assume that $N$ is an elementary abelian $p$-group. Note also that as $N$ is abelian and not semiregular, $N$ is intransitive. Furthermore, since $\Gamma$ is cubic and connected, $G_v$ is a $\{2,3\}$-group, and hence $p\in\{2,3\}$.

Suppose that $p=3$. Since $N$ is not semiregular, we have $N_v \neq 1$ hence $|N_v^{\Gamma(v)}|$ is divisible by $3$ and therefore $N_v^{\Gamma(v)}$ is transitive. Let $u\in \Gamma(v)$.  
Since $G$ is transitive on $\V(\Gamma)$, $N_u^{\Gamma(u)}$ is transitive hence every neighbour of $u$ is in $v^N$. Thus every vertex at distance $2$ from $v$ is in $v^N$. As $N$ is abelian, $N_v$ fixes $v^N$ pointwise and, since $N_v^{\Gamma(v)}$ is transitive, this implies that every neighbour of $v$ has the same neighbourhood. Therefore $\Gamma\cong\K_{3,3}$.

Suppose that $p=2$. Since $N$ is not semiregular, we have $N_v\neq 1$ and hence $|N_v^{\Gamma(v)}|=2$. Since $N_v^{\Gamma(v)}$ is normal in $G_v^{\Gamma(v)}$ this implies that $|G_v^{\Gamma(v)}|=2$. In particular, $v$ has a unique neighbour $v'$ such that $G_v=G_{v'}$. It easily follows that $G$ has two orbits on edges, one of which is $\cT:=\{\{v^g,(v')^g\}\mid g\in G\}$. Note that $\cT$ is a perfect matching of $\Gamma$ and that removing  $\cT$ from the edges of  $\Gamma$ leaves a union of pairwise disjoint cycles of the same length, say $k$. 

Let $u\in \Gamma(v)$ with $u\neq v'$, let $C$ be the cycle of $\Gamma-\cT$ containing $u$ and $v$, and observe that $C$ is a block of imprimitivity for $G$ and hence also for $N$. Note that $N_u$ and $N_v$ act on $C$ as reflections fixing adjacent vertices. Therefore $\langle N_v,N_u\rangle$ fixes $C$ setwise, and the permutation group induced by $\langle N_v,N_u\rangle$ on $C$ is either $\D_k$ (when $k$ is odd) or $\D_{k/2}$ (when $k$ is even). Since $N$ is abelian, it follows that $k=4$.

Suppose that $\Gamma$ is a circular ladder graph, that is, $\Gamma$ is isomorphic to the Cartesian product of a cycle of length $n\geq 3$ with a complete graph on $2$ vertices. If $n=4$ then $\Gamma\cong\Q_3$. We thus assume that $n\neq 4$. In particular, some edges are contained in a unique $4$-cycle while others are contained in more than one $4$-cycle. Call the latter \emph{rungs}. Since $G$ has two orbits on edges and the rungs form a perfect matching, $\cT$ must be the set of rungs. This implies that $\Gamma-\cT$ consists of two cycles of length $n$, contradicting the fact that $k=4$.

Suppose now that $\Gamma$ is a M\"{o}bius ladder graph, that is, $\Gamma$ is isomorphic to the Cayley graph $\Cay(\ZZ_{2n},\{1,-1,n\})$ for some $n\geq 2$. If $n=2$ then $\Gamma\cong\K_4$ and if $n=3$ then $\Gamma\cong\K_{3,3}$.  We thus assume that $n\geq 4$ and the same argument as in the last paragraph yields again that $\cT$ is the set of edges that are contained in more than one $4$-cycle. The removal of these leaves a cycle of length $2n$, which is a contradiction.

We may thus assume that $\Gamma$ is neither a circular ladder nor a M\"{o}bius ladder graph.  From now on, we adopt the terminology of~\cite[Section~4.1]{PSV1280}. By~\cite[Lemma~9]{PSV1280}, it follows that $(\Gamma,G)$ is non-degenerate (that is, for any two edges $\{u,u'\}$ and $\{v,v'\}$ in $\cT$, there is at most one edge of $\Gamma$ between $\{u,u'\}$ and $\{v,v'\}$).

Let $\M(\Gamma,G)$ and $\cC(\Gamma,G)$ be as in~\cite[Construction~7]{PSV1280} (that is, $\M(\Gamma,G)$ is the (not necessarily normal) quotient graph of $\Gamma$ with respect to the vertex-partition $\cT$ and $\cC(\Gamma,G)$ is the image of the cycle decomposition of $\Gamma-\cT$ under the canonical projection to $\M(\Gamma,G)$). By~\cite[Theorem~10]{PSV1280}, $\M(\Gamma,G)$ is a connected $4$-valent $G$-arc-transitive graph and $\cC(\Gamma,G)$ is an arc-transitive cycle decomposition of $\M(\Gamma,G)$ consisting of cycles of length $k=4$. Moreover, by~\cite[Theorem~12]{PSV1280}, $\Gamma\cong\Split(\M(\Gamma,G),\cC(\Gamma,G))$.

Note that $1<N_v\leq N_{\{v,v'\}}$ and thus $N$ is not semiregular on $\M(\Gamma,G)$. By Theorem~\ref{PXu}, $\M(\Gamma,G)\cong\PX(2,r,s)$ for some $r\geq 3$ and $1\leq s\leq r-1$. By Lemma~\ref{boring}, $\cC(\Gamma,G)$ is conjugate to the natural  cycle decomposition of $\M(\Gamma,G)$ under $\Aut(\M(\Gamma,G))$. It follows that $\Split(\M(\Gamma,G),\cC(\Gamma,G))\cong\S(\PX(2,r,s))$, which completes the proof.
\end{proof}

The remaining results in this section are observations about the automorphism group of $\S(\PX(2,r,s))$. They will be useful in the proof of Theorem~\ref{theorem:main1}.

\begin{lemma}\label{AutSplit}
Let $r\geq 5$ and let $1\leq s\leq r-1$. Then $\Aut(\S(\PX(2,r,s))=\ZZ_2^r\rtimes\D_r$ with the permutation representation given in Definition~$\ref{defdef}$.
\end{lemma}
\begin{proof}
Let $\Gamma=\S(\PX(2,r,s))$, let $G=\Aut(\Gamma)$ and let $v$ be a vertex of $\Gamma$. Note that $\Gamma$ is not arc-transitive: some edges are contained in cycles of length four, others are not. 
Let $W=\ZZ_2\wr \D_r=\ZZ_2^r\rtimes\D_r$ act on $\Gamma$ as described in Definition~\ref{defdef}. Since $W\leq G$ and $W_v\neq 1$, it follows that $|G_v^{\Gamma(v)}|=2$.

 We follow the terminology from~\cite[Section~4.1]{PSV1280}. Let $\M(\Gamma,G)$  be as in~\cite[Construction~7]{PSV1280}. Then $\M(\Gamma,G)\cong\PX(2,r,s)$. By~\cite[Lemma~9]{PSV1280}, if $(\Gamma,G)$ is degenerate then every edge of $\Gamma$ is contained in a $4$-cycle, which is not the case. It follows that $(\Gamma,G)$ is not degenerate and thus, by~\cite[Theorem~10]{PSV1280}, $G$ acts faithfully as a group of automorphisms of $\M(\Gamma,G)$, that is, $G\leq\Aut(\M(\Gamma,G))\cong \Aut(\PX(2,r,s))$. By~\cite[Theorem~2.13]{PraegerXu}, $\Aut(\PX(2,r,s))=W$ and thus $W=G$.
\end{proof}

\begin{corollary}\label{PX case}
Let $r\geq 5$, let $1\leq s\leq r-1$ and let $G$ be a vertex-transitive group of automorphisms of $\S(\PX(2,r,s))$. Then $G$ contains a semiregular element of order at least $r$.
\end{corollary}

\begin{proof}
Let $\Gamma=\S(\PX(2,r,s))$. We use the definition of $\S(\PX(2,r,s))$ from Definition~\ref{defdef} so that $\V(\Gamma)=\ZZ_2^s\times \ZZ_r\times \{+,-\}$.  By Lemma~\ref{AutSplit} we have that $\Aut(\Gamma)=\ZZ_2^r\rtimes\D_r$. From Definition~\ref{defdef}, we see that the action of $\ZZ_2^r\rtimes \D_r$ on $\V(\Gamma) $ induces a regular action of $\D_r$ on $\ZZ_r\times \{+,-\}$. 

Let $\pi:\Aut(\Gamma)\to \D_r$ be the natural projection. Since $G$ acts transitively on $\V(\Gamma)$, we obtain that $G$ projects surjectively onto $\D_r$, that is, $\pi(G)=\D_r$. Therefore, $G$ contains an element $g=(g_0,\ldots,g_{r-1},h)$ with $g_0,\ldots,g_{r-1}\in \ZZ_2$ and $h$ an element of order $r$ in $\D_r$.  Clearly, $g$ has order a multiple of $r$ and a computation yields that $g^r=(x,\ldots,x,1)\in\ZZ_2^r\rtimes \D_r$ where $x=g_0+g_1+\cdots+g_{r-1}$. If $x=0$ then $g^r=1$ and $g$ is a semiregular element of order $r$. If $x=1$ then $g^r=(1,\ldots,1,1)$ is a semiregular involution and hence $g$ is semiregular of order $2r$.
\end{proof}

\section{Proof of Theorem~\ref{theorem:main1}}\label{sec:last}


\smallskip

\noindent\textbf{Theorem~\ref{theorem:main1}.}\emph{
There exists a function $f:\mathbb{N}\to \mathbb{N}$ satisfying $f(n)\rightarrow \infty$ as $n\rightarrow\infty$ such that, if $\Gamma$ is a connected $G$-vertex-transitive cubic graph of order $n$ then $G$ contains a semiregular subgroup of order at least $f(n)$.}

\begin{proof}
Our proof uses the abelian normal quotient method and Theorem~\ref{CubicPXu}. We argue by contradiction and hence we begin by assuming that there exists no such function $f$. This means that there exist a constant $c$ and an infinite family $\cF=\{(\Gamma_k,G_k)\}_{k\in\mathbb {N}}$, with $\Gamma_k$ a connected $G_k$-vertex-transitive cubic graph,  such that $\sup\{|\V(\Gamma_k)|\mid k\in\mathbb{N}\}=\infty$ and every semiregular subgroup of $G_k$ has order at most $c$.

For every $k$, let $M_k$ be a normal subgroup of $G_k$ of maximal cardinality subject to $\Gamma_k/M_k$ being cubic and let $\cF^*=\{(\Gamma_k/M_k,G_k/M_k)\}_{k\in\mathbb {N}}$.  Observe that $M_k$ coincides with the kernel of the action of $G_k$ on $M_k$-orbits and that $M_k$ is semiregular. In particular, $|M_k|\leq c$ and moreover, if $H_k/M_k$ is a semiregular subgroup of $G_k/M_k$ in its action on $\V(\Gamma_k/M_k)$, then $H_k$ is semiregular. It follows that $\Gamma_k/M_k$ is a connected $G_k/M_k$-vertex-transitive cubic graph such that $\sup\{|\V(\Gamma_k/M_k)|\mid k\in\mathbb{N}\}=\sup\{|\V(\Gamma_k)|/|M_k|\mid k\in\mathbb{N}\}=\infty$ and every semiregular subgroup of $G_k/M_k$ has order at most $c/|M_k|\leq c$. Replacing $\cF$ by $\cF^*$, we may thus assume that for every nontrivial normal subgroup $M_k$ of $G_k$, the normal quotient $\Gamma_k/M_k$ has valency less than three. 

Replacing $\cF$ by a subfamily, we may also assume that one of the following occurs:

\begin{enumerate}
\item for every $k$, $G_k$ has no nontrivial abelian normal subgroups; 
\item for every $k$, $G_k$ has an abelian normal subgroup that is not semiregular;
\item for every $k$, every abelian normal subgroup of $G_k$ is semiregular and $G_k$ has at least one such subgroup.
\end{enumerate}

\medskip
\noindent\textbf{Case 1.} For every $k$, $G_k$ has no nontrivial abelian normal subgroups. 
\smallskip

\noindent In this case, the socle of $G_{k}$ is a direct product of nonabelian simple groups, that is, $\soc (G_{k})= T_{k,1}\times \cdots\times T_{k,t_k}$, where $T_{k,1},\ldots,T_{k,t_k}$ are nonabelian simple groups. For every $k$ and $j\in \{1,\ldots,t_k\}$, by Burnside's Theorem there exists a prime $p_{k,j}\geq 5$ dividing $|T_{k,j}|$, and hence there exists $x_{k,j}\in T_{k,j}$ with $|x_{k,j}|=p_{k,j}$. Since the stabiliser of a vertex of $\Gamma_{k}$ is a $\{2,3\}$-group, we get that $H_{k}=\langle x_{k,1}\rangle\times\cdots\times\langle x_{k,t_k}\rangle$ is a semiregular subgroup of $G_{k}$ of order $\prod_{j}p_{k,j}\geq 5^{t_k}$. Thus $t_k\leq \log_5(c)$.

Using  the CFSG, it can be shown that there exists a function $g:\mathbb{N}\to \mathbb{N}$  satisfying $g(n)\rightarrow \infty$ as $n\rightarrow\infty$ such that if $T$ is a nonabelian simple group of  order $n$ then $T$ contains an element $t$ of order at least $g(n)$ and coprime to $6$ (see for example~\cite[Lemma~$3.5$]{SpigaCubic}). Since $T_{k,j}$ has no element of order larger than $c$ and coprime to $6$, we get $g(|T_{k,j}|)\leq c$. It follows that there exists a constant $b$ such that $|T_{k,j}|\leq b$ for every $k$ and $j\in \{1,\ldots,t_k\}$.

We have shown that $|\soc(G_{k})|\leq b^{\log_5(c)}$ for every $k$. As the action of $G_{k}$ on $\soc(G_{k})$ by conjugation is faithful, $G_{k}$ is isomorphic to a subgroup of $\Aut(\soc(G_{k}))$ and  hence, since $G_k$ is vertex-transitive, $|\V(\Gamma_k)| \leq |G_{k}|\leq |\Aut(\soc(G_{k}))|\leq (b^{\log_5(c)})!$. This contradicts the fact that $\sup\{|\V(\Gamma_k)|\mid k\in\mathbb{N}\}=\infty$.

\medskip

\noindent\textbf{Case 2.} For every $k$, $G_k$ has an abelian normal subgroup that is not semiregular.

\smallskip

\noindent  Replacing $\cF$ by a subfamily, we may assume that $|\V(\Gamma_k)|>32$ for every $k$. By Theorem~\ref{CubicPXu}, it follows that $\Gamma_{k}$ is isomorphic to $\S(\PX(2,r_k,s_k))$ for some $r_k\geq 5$ and $1\leq s_k\leq r_k-1$. Now, from Corollary~\ref{PX case} we get $r_k\leq c$ and hence $|\V(\Gamma_{k})|=2^{s_k}r_k\leq 2^{c-1}c$. This contradicts the fact that $\sup\{|\V(\Gamma_k)|\mid k\in\mathbb{N}\}=\infty$.

\medskip

\noindent\textbf{Case 3.} For every $k$, every abelian normal subgroup of $G_k$ is semiregular and $G_k$ has at least one such subgroup.

\smallskip

\noindent Replacing $\cF$ by a subfamily, we may assume that $|\V(\Gamma_k)|>2c$ for every $k$. Let $N_k$ be an abelian minimal normal subgroup of $G_k$. Note that $N_k$ is elementary abelian and semiregular and hence $|N_k|\leq c$. Since $|\V(\Gamma_k)|>2c$, it follows that $N_k$ has at least three orbits and, since $N_k\neq 1$, the graph $\Gamma_k/N_k$ has valency at most two and hence is a cycle of length $|\V(\Gamma_k)|/|N_k|\geq |\V(\Gamma_k)|/c$.

Let $K_k$ be the kernel of the action of $G_k$ on $N_k$-orbits and let $C_k$ be the centraliser of $N_k$ in $K_k$. As $N_k$ is abelian, we have $N_k\leq C_k$. Also, as $N_k$ and $K_k$ are normal in  $G_k$, so is $C_k$. Since $N_k$ is abelian and $K_k$ preserves the $N_k$-orbits setwise, we must have $C_k^\Delta=N_k^\Delta$ for each $N_k$-orbit $\Delta$. It follows that the commutator $[C_k,C_k]$ fixes each $N_k$-orbit pointwise and hence $[C_k,C_k]=1$. Thus $C_k$ is abelian and hence semiregular. For $v\in \V(\Gamma_k)$, we have $K_k=N_k(K_k)_v$. As $N_k\leq C_k\leq K_k$, this implies that $C_k=N_k$, that is, $\cent {K_k}{N_k}=N_k$. 

Since $|N_k|\leq c$, we have $|G_k:\cent {G_k}{N_k}|\leq |\Aut(N_k)|\leq c!$. Thus $|G_k/K_k:K_k\cent {G_k}{N_k}/K_k|\leq c!$. Recall that $G_k/K_k$ acts faithfully and vertex-transitively on the cycle $\Gamma_k/N_k$ and thus contains a $2$-step rotation. Since $|G_k/K_k:K_k\cent {G_k}{N_k}/K_k|\leq c!$, it follows that $\cent {G_k}{N_k}$ contains an element $g_k$ acting as an $\ell_k$-step rotation of $\Gamma_k/N_k$ with $\ell_k\leq (2c!)$. Now, $g_k^{\ell_k}\in K_k\cap \cent {G_k}{N_k}=\cent {K_k}{N_k}=N_k$ and hence $g_k^{\ell_k}$ is semiregular, and so is $g_k$.  It follows that $\langle g_k\rangle$ is a semiregular subgroup of $G_k$ of order at least $|\V(\Gamma_k/N_k)|/(2c!)\geq|\V(\Gamma_k)|/(2cc!)$. Since $\sup\{|\V(\Gamma_k)|\mid k\in\mathbb{N}\}=\infty$, this is our final contradiction.
\end{proof}

\end{document}